\documentclass[12pt,a4paper]{amsart}
\allowdisplaybreaks[1]

\usepackage{amstext}
\usepackage{amsthm}
\usepackage{amssymb}

\newtheorem{thm}{Theorem}[section]
\newtheorem{lem}[thm]{Lemma}
\newtheorem{prop}[thm]{Proposition}

\theoremstyle{definition}
\newtheorem{defn}[thm]{Definition}

\theoremstyle{remark}
\newtheorem{rem}[thm]{Remark}

\DeclareMathOperator{\dep}{dep}
\DeclareMathOperator{\wt}{wt}

\begin{document}

\title[A combinatorial proof of the weighted sum formula]{A combinatorial proof of the weighted sum formula for finite and symmetric multiple zeta(-star) values}

\author{Hideki Murahara}
\address[Hideki Murahara]{Nakamura Gakuen University Graduate School, 5-7-1, Befu, Jonan-ku,
Fukuoka, 814-0198, Japan}
\email{hmurahara@nakamura-u.ac.jp}


\begin{abstract}
 Hirose, Saito, and the author established the weighted sum formula for
 finite multiple zeta(-star) values.
 In this paper, we present its alternative proof.
 The proof is also valid for symmetric multiple zeta(-star) values.
\end{abstract}

\maketitle

\section{Introduction}
For positive integers $k_{1},\dots,k_{r}$ with $k_{r}\ge2$, 
the multiple zeta values (MZVs) and the multiple zeta-star values
(MZSVs) are defined by 
\begin{align*}
 \zeta(k_{1},\dots,k_{r})
 &:=\sum_{0<n_{1}<\cdots<n_{r}}\frac{1}{n_{1}^{k_{1}}\cdots n_{r}^{k_{r}}} \in\mathbb{R},\\
 \zeta^{\star}(k_{1},\dots,k_{r})
 &:=\sum_{0<n_{1}\le\cdots\le n_{r}}\frac{1}{n_{1}^{k_{1}}\cdots n_{r}^{k_{r}}} \in\mathbb{R}.
\end{align*}
We set a $\mathbb{Q}$-algebra $\mathcal{A}$ by
\[
 \mathcal{A}:=\biggl(\prod_{p}\mathbb{Z}/p\mathbb{Z}\biggr)\,\Big/\,\biggl(\bigoplus_{p}\mathbb{Z}/p\mathbb{Z\biggr)},
\]
where $p$ runs over all primes.
For positive integers $k_{1},\dots,k_{r}$,
the finite multiple zeta values (FMZVs) and the finite multiple zeta-star
values (FMZSVs) are defined by 
\begin{align*}
 \zeta_{\mathcal{A}}(k_{1},\dots,k_{r})
 &:=\biggl(\sum_{0<n_{1}<\cdots<n_{r}<p} \frac{1}{n_{1}^{k_{1}}\cdots n_{r}^{k_{r}}}\bmod p \biggr)_{p} \in\mathcal{A}, \\
 \zeta_{\mathcal{A}}^{\star}(k_{1},\dots,k_{r})
 &:=\biggl(\sum_{0<n_{1}\le\cdots\le n_{r}<p} \frac{1}{n_{1}^{k_{1}}\cdots n_{r}^{k_{r}}}\bmod p \biggr)_{p} \in\mathcal{A}.
\end{align*}
The symmetric multiple zeta values (SMZVs) were introduced by Kaneko-Zagier
\cite{KZ17}. 
For positive integers $k_{1},\dots,k_{r}$,
we define 
\[
 \zeta_{\mathcal{S}}^{\ast}(k_{1},\dots,k_{r}) 
 :=\sum_{i=0}^{r}(-1)^{k_{i+1}+\cdots+k_{r}}\zeta^{\ast}(k_{1},\dots,k_{i})\zeta^{\ast}(k_{r},\dots,k_{i+1})\in\mathbb{R}.
\]
Here, the symbol $\zeta^{\ast}$ on the right-hand side means 
the regularized value coming from harmonic regularization, i.e.\ a
real value obtained by taking constant terms of harmonic regularization
as explained in Ihara-Kaneko-Zagier \cite{IKZ06}. In the sum, we
understand $\zeta^{\ast}(\emptyset)=1$. Let $\mathcal{Z}_{\mathbb{R}}$
be the $\mathbb{Q}$-vector subspace of $\mathbb{R}$ spanned by $1$
and all MZVs, which is a $\mathbb{Q}$-algebra. Then the SMZVs is
defined by 
\[
 \zeta_{\mathcal{S}}(k_{1},\dots,k_{r}):=\zeta_{\mathcal{S}}^{\ast}(k_{1},\dots,k_{r})\bmod\zeta(2)\in\mathcal{Z}_{\mathbb{R}}/(\zeta(2)).
\]
For positive integers $k_{1},\dots,k_{r}$, we also define the symmetric
multiple zeta-star values (SMZSVs) by 
\[
 \zeta_{\mathcal{S}}^{\star}(k_{1},\dots,k_{r})
 :=\sum_{\substack{\square\textrm{ is either a comma `,' }\\
 \textrm{ or a plus `+'}
 }
 }\zeta_{\mathcal{S}}^{\ast}(k_{1}\square\cdots\square k_{r})\bmod\zeta(2)\in\mathcal{Z}_{\mathbb{R}}/(\zeta(2)).
\]
Let $\mathcal{Z}_{A}$ denote the $\mathbb{Q}$-vector subspace of $\mathcal{A}$ spanned by $1$ and all FMZVs. 
Kaneko and Zagier conjecture
that there is an isomorphism between $\mathcal{Z}_{A}$ and $\mathcal{Z}_{\mathbb{R}}/(\zeta(2))$
as $\mathbb{Q}$-algebras such that $\zeta_{\mathcal{A}}(k_{1},\dots,k_{r})$
and $\zeta_{\mathcal{S}}(k_{1},\dots,k_{r})$ correspond with each
other (for more details, see Kaneko-Zagier \cite{Kan19,KZ17}). 
In the following, we use the letter $\mathcal{F}$ stands for either $\mathcal{A}$ or $\mathcal{S}$,
e.g., the symbol $\zeta_\mathcal{F}$ means $\zeta_\mathcal{A} \textrm{ or } \zeta_\mathcal{S}$. 

Hirose, Saito, and the author \cite{HMS18} proved the following weighted sum formula
for FMZ(S)Vs. In this paper, we will give its alternative proof. 
Our proof is also valid for SMZ(S)Vs. 
\begin{thm} \label{main}
 Let $k$ be a positive integer, $r$ a positive odd integer, and $i$ an integer with $1\le i \le r\le k$. 
 Then we have
 \begin{align*}
  \sum_{\substack{k_{1}+\cdots+k_{r}=k\\k_1,\dots,k_r\ge1}}
  2^{k_{i}}\,\zeta_{\mathcal{F}}(k_{1},\ldots,k_{r}) & =0,\\
  \sum_{\substack{k_{1}+\cdots+k_{r}=k\\k_1,\dots,k_r\ge1}}
  2^{k_{i}}\,\zeta_{\mathcal{F}}^{\star}(k_{1},\ldots,k_{r}) & =0.
 \end{align*}
\end{thm}
\begin{rem}
 We note that similar weighted sum formulas for MZVs are known (see Guo-Xie \cite{GLX09}, Ohno-Zudilin \cite{OZ08}, and Ong-Eie-Liaw \cite{OEL13}). 
 Kamano \cite{Kam18} also obtained somewhat different weighted sum formulas for FMZVs.
\end{rem}

%%%%%%%%%%%%%%%%%%%%%%%%%%%%%%%%%%%%%%%%%%%%%%%%%%%%%%%%%%%%%%%%%%%%%%%%%%%%%%%%%%%%%%%%%%%
\section{Proof of the main theorem}
\subsection{Notation}
An index is a sequence of positive integers, and
we denote by $\mathcal{I}$ the $\mathbb{Q}$-linear space spanned by the indices. 
Maps defined for indices, such as $\zeta_\mathcal{F}$, will be extended $\mathbb{Q}$-linearly.
For an index $\boldsymbol{k}=(k_{1},\ldots,k_{r})$, 
the integer $k:=k_{1}+\cdots+k_{r}$ is called the weight of $\boldsymbol{k}$
(denoted by $\wt(\boldsymbol{k})$) and the integer $r$ is called
the depth of $\boldsymbol{k}$ (denoted by $\dep(\boldsymbol{k})$).
We write $(\{1\}^m)=(\underbrace{1,\ldots ,1}_{m})$ for nonnegative integer $m$.
For indices $\boldsymbol{k}=(k_{1}\ldots,k_{r})$ and $\boldsymbol{k}'=(k'_{1},\ldots,k'_{r})$ of the same depths,
the symbol $\boldsymbol{k}\oplus\boldsymbol{k}'$ represents the componentwise sum, i.e.\ 
$\boldsymbol{k}\oplus\boldsymbol{k}':=(k_{1}+k'_{1},\ldots,k_{r}+k'_{r})$.
\begin{defn}
 For an index $\boldsymbol{k}=(k_{1}\ldots,k_{r})$, we define $\phi(\boldsymbol{k})$ by
 \[
  \phi(\boldsymbol{k})
  :=(-1)^r \sum_{\substack{\square\textrm{ is either a comma `,' } \\ \textrm{ or a plus `+'}} }
   (\underbrace{1\square1\square\cdots\square 1}_{\substack{ \textrm{the number of} \\ \textrm{`1' is }k_1}}
   ,\cdots
   ,\underbrace{1\square1\square\cdots\square 1}_{\substack{ \textrm{the number of} \\ \textrm{`1' is }k_r}})
  \in\mathcal{I}.
 \]
 For example, we have $\phi(1,2,2)=-(1,2,2)-(1,1,1,2)-(1,2,1,1)-(1,1,1,1,1)$.
\end{defn}
\begin{defn}
 For a nonempty index $\boldsymbol{k}=(k_1,\dots,k_r)$, we define Hoffman's dual index of $\boldsymbol{k}$ by
 \begin{align*}
  \boldsymbol{k}^{\vee}=(\underbrace{1,\dots,1}_{k_1}+\underbrace{1,\dots,1}_{k_2}+1,\dots,1+\underbrace{1,\dots,1}_{k_r}).
 \end{align*}
\end{defn}
For positive integers $k,r,i$ with $1\le i\le r\le k$, we set 
\begin{align*}
 F(k,r,i)
 :=\sum_{\substack{k_{1}+\cdots+k_{r}=k\\k_1,\dots,k_r\ge1}}
  2^{k_{i}-1} \cdot (k_1,\ldots,k_r) \in\mathcal{I}.
\end{align*}
Throughout this paper, we always assume that $\boldsymbol{e}$ runs over sequences of nonnegative integers.
(Recall that an index is a sequence of positive integers.) 
For an index $\boldsymbol{k}$ and a nonnegative integer $l$, we also set  
\begin{align*}
 G_{1}(\boldsymbol{k},l)
 &:=\sum_{\substack{\wt(\boldsymbol{e})=l\\ \dep(\boldsymbol{e})=\dep(\boldsymbol{k})}}
  (\boldsymbol{k}\oplus\boldsymbol{e}) \in\mathcal{I},\\
 G_{2}(\boldsymbol{k},l)
 &:=\sum_{\substack{\wt(\boldsymbol{e})=l\\ \dep (\boldsymbol{e})=\dep(\boldsymbol{k}^\vee)}}
  (\boldsymbol{k}^{\vee}\oplus\boldsymbol{e})^{\vee} \in\mathcal{I},\\
 G(\boldsymbol{k},l)
 &:=G_{1}(\boldsymbol{k},l)-G_{2}(\boldsymbol{k},l) \in\mathcal{I}.
\end{align*}
For example, we have $G_{1}((2,3),1)=(3,3)+(2,4)$.
For positive integers $k,r,i$ with $1\le i\le r\le k$, we put
\begin{align*}
 &H(k,r,i) \\
 &:=F(k,r,i) \\
 &\quad
  -\left(\sum_{l=1}^{k-r-1}
  2^{l-1} G((\{1\}^{i-1},l+1,\{1\}^{r-i}),k-r-l) 
  +G((\{1\}^{r}),k-r)\right) \in\mathcal{I}.
\end{align*}
\subsection{Proof of Theorem \ref{main}}
To prove Theorem \ref{main}, we use Theorems \ref{phi} and \ref{ohnoF}, and Lemma \ref{key_lem}. 
\begin{thm}[Hoffman \cite{Hof15}, Jarossay \cite{Jar14}] \label{phi}
 For an index $\boldsymbol{k}$, we have 
 \begin{align*}
  \zeta_{\mathcal{F}}(\boldsymbol{k}) & =\zeta_{\mathcal{F}}(\phi(\boldsymbol{k})).
 \end{align*}
\end{thm}
The following relation for FMZVs and SMZVs was conjectured by Kaneko and established by Oyama \cite{Oya15}.
\begin{thm}[Oyama \cite{Oya15}] \label{ohnoF}
 For an index $\boldsymbol{k}$ and a nonnegative integer $l$, we have 
 \[
  \zeta_{\mathcal{F}}(G(\boldsymbol{k},l))=0.
 \]
\end{thm}
\begin{lem}[Key lemma] \label{key_lem}
 Let $k$ be a positive integer, $r$ a positive odd integer, and $i$ an integer with $1\le i \le r\le k$. 
 Then we have
 \begin{align*}
  H(k,r,i)+\phi(H(k,r,i))=
  \begin{cases}
   \,(\{1\}^{k})\quad & (k\textrm{:even}),\\
   \quad0 & (k\textrm{:odd}).
  \end{cases}
 \end{align*}
\end{lem}
To prove Lemma \ref{key_lem}, we need Lemmas \ref{lem1} and \ref{lem2}.
\begin{lem} \label{lem1}
 Let $k,r$ be positive integers and $i$ an integer with $1\le i \le r\le k$. 
 Then we have
 \begin{align*}
  F(k,r,i)
   -\left(\sum_{l=1}^{k-r-1}
  2^{l-1} G_1((\{1\}^{i-1},l+1,\{1\}^{r-i}),k-r-l) 
  +G_1((\{1\}^{r}),k-r)\right) \\
  =2^{k-r-1}\,(\{1\}^{i-1},k-r+1,\{1\}^{r-i}).
 \end{align*}
\end{lem}
\begin{proof}
 We note that the depths of all the indices on the left-hand side are $r$.
 Put
 \begin{align*}
  A:=\sum_{l=1}^{k-r-1}
  2^{l-1} G_1((\{1\}^{i-1},l+1,\{1\}^{r-i}),k-r-l) 
  +G_1((\{1\}^{r}),k-r). 
 \end{align*}
 Fix an index $(a_1,\ldots,a_i,\ldots,a_r)$ of weight $k$ and depth $r$. 
 Since the number of indices $(a_1,\ldots,a_i,\ldots,a_r)$ included in $A$ is 
 \begin{align*}
  \begin{cases}
   \displaystyle{
    \sum_{l=1}^{a_i-1} 2^{l-1}+1=2^{a_i-1} \quad (1\le a_i\le k-r), 
   } \\
   \displaystyle{
    \sum_{l=1}^{k-r-1} 2^{l-1}+1=2^{k-r-1} \quad (a_i=k-r+1), 
   } \\
  \end{cases}
 \end{align*}
 we find that all the indices appeared in $F(k,r,i)$ and $A$ are same except for  $2^{k-r-1}\,(\{1\}^{i-1},k-r+1,\{1\}^{r-i})$.
 Thus we find the result. 
\end{proof}
\begin{lem} \label{lem2}
 Let $k$ be a positive integer, $r$ a positive odd integer, and $i$ an integer with $1\le i \le r\le k$. 
 Then we have
 \begin{align*}
  &\left(\sum_{l=1}^{k-r-1}
  2^{l-1} G_2((\{1\}^{i-1},l+1,\{1\}^{r-i}),k-r-l) 
  +G_2((\{1\}^{r}),k-r)\right) \\
  &+\phi
  \left(\sum_{l=1}^{k-r-1}
  2^{l-1} G_2((\{1\}^{i-1},l+1,\{1\}^{r-i}),k-r-l) 
  +G_2((\{1\}^{r}),k-r)\right) \\
  &=
  \begin{cases}
   -2^{k-r-1}\,( (\{1\}^{i-1},k-r+1,\{1\}^{r-i}) + \phi (\{1\}^{i-1},k-r+1,\{1\}^{r-i}) )
   +(\{1\}^{k}) \\
   \,\,\,\qquad\qquad\qquad\qquad\qquad\qquad\qquad\qquad\qquad\qquad\qquad\qquad\qquad\qquad (k\textrm{:even}),\\
   -2^{k-r-1}\,( (\{1\}^{i-1},k-r+1,\{1\}^{r-i}) + \phi (\{1\}^{i-1},k-r+1,\{1\}^{r-i}) ) \\
   \,\quad\qquad\qquad\qquad\qquad\qquad\qquad\qquad\qquad\qquad\qquad\qquad\qquad  (k\textrm{:odd}).
  \end{cases}
 \end{align*}
\end{lem}
\begin{proof}
 Put
 \begin{align*}
  B
  &:=\sum_{l=1}^{k-r-1}
  2^{l-1} G_2((\{1\}^{i-1},l+1,\{1\}^{r-i}),k-r-l) 
  +G_2((\{1\}^{r}),k-r), \\
  C
  &:=-2^{k-r-1}\,( (\{1\}^{i-1},k-r+1,\{1\}^{r-i}) + \phi (\{1\}^{i-1},k-r+1,\{1\}^{r-i})).
 \end{align*}
 Then we have
 \begin{align*}
  B
  &=\sum_{l=1}^{k-r-1} 
  2^{l-1} \sum_{\substack{ \wt (\boldsymbol{e})=k-r-l \\ \dep (\boldsymbol{e})=l+1 }}
  (((i,\{1\}^{l-1},r-i+1)\oplus\boldsymbol{e})^\vee) 
   +(\{1\}^k) \\
  &=\sum_{l=1}^{k-r-1} 
   2^{l-1} \sum_{\substack{ k_1+\cdots+k_{k-r-l+1}=k-r+1 \\ k_1,\dots,k_{k-r-l+1}\ge1 }}
   (\{1\}^{i-1},k_1,\ldots,k_{k-r-l+1},\{1\}^{r-i}) 
   +(\{1\}^k). 
 \end{align*}
 %We note that the number of commas `,' in the indices $(a_1,\ldots,a_d)$ and $(k_1,\ldots,k_{k-r-l+1})$ are $d-1$ and $k-r-l$, respectively. 
 Fix an index $(\{1\}^{i-1}, a_1,\ldots,a_d, \{1\}^{r-i})$ of weight $k$ and depth $d+r-1$ with $2\le d\le k-r$. 
 Then the numbers of indices $(\{1\}^{i-1}, a_1,\ldots,a_d, \{1\}^{r-i})$ included in $B$ and $\phi(B)$ are, respectively, 
 \[
  2^{k-r-d} \quad \textrm{ and }\quad \sum_{l=k-r-d+1}^{k-r-1} \,2^{l-1}\cdot (-1)^{k-l}\cdot \binom{d-1}{k-r-l}.
 \]
 We have 
 \begin{align*}
  &2^{k-r-d}+\sum_{l=k-r-d+1}^{k-r-1} \,2^{l-1}\cdot (-1)^{k-l}\cdot \binom{d-1}{k-r-l}\\
  &=2^{k-r-d}+2^{k-r-d}\cdot(-1)^{r+d+1} \cdot \sum_{l=0}^{d-2} (-2)^l \binom{d-1}{l} \\
  &=2^{k-r-d}\cdot(-1)^d\cdot \left( (-1)^d+\sum_{l=0}^{d-2} (-2)^l \binom{d-1}{l} \right) \quad (\textrm{by } r\textrm{:odd}).
 \end{align*}
 Since
 \[
  (1-2x)^n=\sum_{a=0}^{n} (-2x)^a \cdot \binom{n}{a},
 \]
 we have
 \[
  \sum_{l=0}^{d-2} (-2)^l \binom{d-1}{l}=(-1)^{d-1}-(-2)^{d-1}
 \]
 by substituting $x=1$, $a=l$, and $n=d-1$. 
 Thus we get
 \begin{align*}
  &2^{k-r-d}+\sum_{l=k-r-d+1}^{k-r-1} \,2^{l-1}\cdot (-1)^{k-l}\cdot \binom{d-1}{k-r-l}\\
  &=2^{k-r-d}\cdot(-1)^{d} \cdot \left( (-1)^d +(-1)^{d-1}-(-2)^{d-1} \right) \\
  &=2^{k-r-1}.  
 \end{align*}
 Similarly, the number of indices $(\{1\}^{k})$ in $B+\phi(B)$ is
 \begin{align*}
  &1+\sum_{l=1}^{k-r-1} \,2^{l-1}\cdot (-1)^{k-l}\cdot \binom{k-r}{k-r-l} +(-1)^k \\
  &=1+\frac{1}{2} (-1)^k \sum_{l=1}^{k-r-1} (-2)^l \binom{k-r}{l} +(-1)^k \\
  &=1+\frac{1}{2} (-1)^k \left( (-1)^{k-r} -1-(-2)^{k-r} \right) +(-1)^k \\
  &=2^{k-r-1} +\frac{1}{2} (1+(-1)^k)  \qquad (\textrm{by } r\textrm{:odd}) \\
  &=
   \begin{cases}
    2^{k-r-1}+1 \quad & (k\textrm{:even}),\\
    2^{k-r-1} \quad & (k\textrm{:odd}).
   \end{cases}
 \end{align*}
 Then we find
 \begin{align*}
  &B+\phi(B) \\
  &=\sum_{d=2}^{k-r} \sum_{ \substack{ a_1+\cdots+a_d=k-r+1 \\ a_1,\dots,a_d\ge1 } } 
   2^{k-r-1} (\{1\}^{i-1}, a_1,\ldots,a_d, \{1\}^{r-i}) \\
  &\quad 
  +\begin{cases}
    (2^{k-r-1}+1) (\{1\}^k) \quad & (k\textrm{:even}),\\
    2^{k-r-1} (\{1\}^k) \quad & (k\textrm{:odd}). 
   \end{cases}
 \end{align*}
 On the other hand, by the direct caluclation, we have 
 \[
  C=\sum_{d=2}^{k-r+1} \sum_{ \substack{ a_1+\cdots+a_d=k-r+1 \\ a_1,\dots,a_d\ge1 } } 
   2^{k-r-1} (\{1\}^{i-1}, a_1,\ldots,a_d, \{1\}^{r-i}).
 \]
 This finishes the proof.  
\end{proof}
\begin{proof}[Proof of Lemma \ref{key_lem}]
 By Lemmas \ref{lem1} and \ref{lem2}, we easily find the lemma holds. 
\end{proof}
\begin{proof}[Proof of Theorem \ref{main} (the first statement)]
 By Theorems \ref{phi} and \ref{ohnoF}, and Lemma \ref{key_lem}, we see the theorem holds.
\end{proof}

Now we prove the second statement of Theorem \ref{main}. 
We can prove this in the same manner as in Hirose-Murahara-Saito \cite{HMS18}.
The following formulas are well known (see e.g., Sakugawa-Seki \cite{SS17}).
\begin{lem} \label{antipode} 
 For positive integers $k_{1},\dots,k_{r}$, we have
 \[ 
  \sum_{l=0}^{r}(-1)^l \zeta_\mathcal{F}^{\star}(k_{1},\dots,k_{l})\zeta_\mathcal{F} (k_{r},\dots,k_{l+1})=0.
 \]
 Here, we understand $\zeta_\mathcal{F}(\emptyset)=\zeta_\mathcal{F}^{\star}(\emptyset)=1$.
\end{lem}
\begin{prop}[Hoffman \cite{Hof15}, Murahara \cite{Mur15}] \label{sym_sum} 
 For positive integers $k_{1},\dots,k_{r}$, we have
 \begin{align*} 
  \sum_{\sigma\in\mathfrak{S}_r} \zeta_\mathcal{F} (k_{\sigma(1)},\ldots,k_{\sigma(r)})=0, \\
  \sum_{\sigma\in\mathfrak{S}_r} \zeta_\mathcal{F}^{\star} (k_{\sigma(1)},\ldots,k_{\sigma(r)})=0,
 \end{align*}
 where $\mathfrak{S}_r$ is the symmetric group of degree $r$. 
\end{prop}
\begin{proof}[Proof of Theorem \ref{main} (the second statement)]
 By Lemma \ref{antipode}, we have
 \[
  \sum_{l=0}^r (-1)^l 
  \sum_{\substack{k_{1}+\cdots+k_{r}=k\\k_1,\dots,k_r\ge1}}
  2^{k_{i}}\,\zeta_{\mathcal{F}}^{\star}(k_{1},\ldots,k_{l}) \zeta_{\mathcal{F}}(k_{r},\ldots,k_{l+1})
 =0.
 \]
 By using the first statement of Theorem \ref{main} and Proposition \ref{sym_sum}, we find the result.
\end{proof}

%%%%%%%%%%%%%%%%%%%%%%%%%%%%%%%%%%%%%%%%%%%%%%%%%%%%%%%%%%%%%%%%%%%%%%%%%%%%%%%%%%%%%%%%%%%
\section*{Acknowledgement}
The author would like to express his sincere gratitude to 
Doctor Minoru Hirose, Takuya Murakami, and Professor Shingo Saito 
for valuable comments. 
He would also like to thank Doctor Masataka Ono 
for carefully reading the manuscript.

%%%%%%%%%%%%%%%%%%%%%%%%%%%%%%%%%%%%%%%%%%%%%%%%%%%%%%%%%%%%%%%%%%%%%%%%%%%%%%%%%%%%%%%%%%%

\end{document}